\documentclass[preprint,3p,11pt]{elsarticle}
\usepackage{url}
\usepackage{amssymb}
\usepackage{amsmath}
\usepackage{stmaryrd}
\usepackage{siunitx}
\usepackage{commath}
\usepackage{multirow}
\usepackage{booktabs}
\usepackage{subfig}
\usepackage{enumerate}
\usepackage{algorithm}
\usepackage{algpseudocode}
\usepackage[hidelinks]{hyperref}
\hypersetup{
  colorlinks   = true, 
  urlcolor     = red, 
  linkcolor    = red, 
  citecolor   = red 
}
\usepackage{cleveref}
\usepackage{tikz}
\usetikzlibrary{math}
\usetikzlibrary{patterns}

\usepackage{pgfplots}
\usetikzlibrary{intersections, pgfplots.fillbetween}

\usepackage{xspace,color}

\journal{}
\makeatletter
\newcommand{\tnorm}{\@ifstar\@tnorms\@tnorm}
\newcommand{\@tnorms}[1]{%
  \left|\mkern-1.5mu\left|\mkern-1.5mu\left|
   #1
  \right|\mkern-1.5mu\right|\mkern-1.5mu\right|
}
\newcommand{\@tnorm}[2][]{%
  \mathopen{#1|\mkern-1.5mu#1|\mkern-1.5mu#1|}
  #2
  \mathclose{#1|\mkern-1.5mu#1|\mkern-1.5mu#1|}
}
\makeatother

\newcommand{\jump}[1]{\llbracket #1 \rrbracket}

\newtheorem{proposition}{Proposition}
\newdefinition{remark}{Remark}
\newproof{proof}{Proof}
\begin{document}
\begin{frontmatter}
  \title{An interface-tracking space-time hybridizable/embedded
    discontinuous Galerkin method for nonlinear free-surface flows}
  \author{Giselle Sosa Jones\corref{cor1}\fnref{label1}}
  \ead{ggsosajo@central.uh.edu}    
  \address{Department of Mathematics, University of
  	Houston, USA}
  \author{Sander Rhebergen\fnref{label2}}
  \ead{srheberg@uwaterloo.ca}  
  \address{Department of Applied Mathematics, University of
    Waterloo, Canada}

  \cortext[cor1]{Corresponding author}
  \fntext[label1]{\url{https://orcid.org/0000-0001-7505-8256}}
  \fntext[label2]{\url{https://orcid.org/0000-0001-6036-0356}}
  \begin{abstract}
    We present a compatible space-time hybridizable/embedded
    discontinuous Galerkin discretization for nonlinear free-surface
    waves. We pose this problem in a two-fluid (liquid and gas) domain
    and use a time-dependent level-set function to identify the sharp
    interface between the two fluids. The incompressible two-fluid
    equations are discretized by an exactly mass conserving space-time
    hybridizable discontinuous Galerkin method while the level-set
    equation is discretized by a space-time embedded discontinuous
    Galerkin method. Different from alternative discontinuous Galerkin
    methods is that the embedded discontinuous Galerkin method results
    in a continuous approximation of the interface. This, in
    combination with the space-time framework, results in an
    interface-tracking method without resorting to smoothing
    techniques or additional mesh stabilization terms.
  \end{abstract}
  \begin{keyword}
    Navier--Stokes \sep hybridizable \sep discontinuous
    Galerkin \sep space-time \sep free-surface waves \sep
    interface-tracking
  \end{keyword}
\end{frontmatter}
\section{Introduction}
\label{sec:introduction}

Free-surface problems arise in many real-world applications such as in
the design of ships and offshore structures, modeling of tsunamis, and
dam breaking. Mathematically, nonlinear free-surface wave problems are
described by a set of partial differential equations that govern the
movement of the fluid together with certain nonlinear boundary
conditions that describe the free-surface. Modeling such problems is
challenging because the boundary of the computational domain depends
on the solution of the problem. This implies that there is a strong
coupling between the fluid and the free-surface, and the domain must be
continuously updated to track the changes in the free-surface.

Discontinuous Galerkin (DG) finite element methods posses excellent
conservation and stability properties, and provide higher-order
accurate approximate solutions. DG methods are therefore well suited
for the spatial discretization of free-surface problems. Hybridizable
discontinuous Galerkin (HDG) methods possess the same properties as
the DG method, but at a lower computational cost \cite{Kirby:2012,
  Yakovlev:2016, Cockburn:2009a}. This is achieved by reducing the
number of globally coupled degrees-of-freedom. In HDG methods,
so-called facet variables are introduced that live only on the facets
of the mesh. The HDG method is then constructed such that
communication between element unknowns is done only through this facet
variable. This construction allows for static condensation which
significantly reduces the total number of globally coupled
degrees-of-freedom; it is possible to eliminate the element
degrees-of-freedom and solve a linear system only for the facet
degrees-of-freedom. Given the facet degrees-of-freedom, the element
variables can be reconstructed element-wise. For incompressible flows
the HDG method can be constructed such that the approximate velocity
is $H(\text{div})$-conforming and point-wise divergence free on the
elements (see for example \cite{Fu:2019, Lehrenfeld:2016,
  Rhebergen:2018a}). We remark that to lower the number of globally
coupled degrees-of-freedom even further one may impose that the facet
variables are continuous between facets resulting in the embedded
discontinuous Galerkin (EDG) method \cite{Cockburn:2009b, Guzey:2007,
  Rhebergen:2020}.

Free-surface problems can be viewed as two-fluid, e.g. liquid and gas,
flow problems and so level-set methods can be used for their
discretization. Level-set methods were first introduced in
\cite{Osher:1988} and have since been used for multi-phase flows
\cite{Chang:1996, Sussman:1994}, and for free-surface flows
\cite{Hesthaven:2006, Lin:2005, Marchandise:2006}. Using the two-fluid
approach, the fluid equations are solved in a domain that contains
both liquid and gas phases. In this set-up, the viscosity and density
are defined as piece-wise constant functions that are discontinuous
across the liquid-gas interface. A level-set function, which satisfies
an advection equation where the advection field is the velocity of the
fluid, is defined such that it is positive in one of the fluids and
negative in the other. The zero level-set then corresponds to the
interface between the two fluids. Traditionally, when using level-set
methods, the mesh remains fixed throughout the computation. This
results in there being mesh elements in which the density and
viscosity take two different values. To handle this, the density and
viscosity of both fluids are smoothed throughout a band around the
interface. This results in non-physical fluid properties, and
introduces an extra parameter (the thickness of the band) in the
discretization of the problem. Instead, in this paper we consider an
interface-tracking method by solving the time dependent advection
equation for the level-set function and moving the mesh to fit the
zero level-set. Since every mesh element will belong to only one of
the fluids' domains, no smoothing is necessary of the density and
viscosity. Two well-known approaches to handle moving meshes are the
Arbitrary Lagrangian-Eulerian (ALE) approach \cite{Labeur:2009,
  Fu:2020, Neunteufel:2020} and the space-time approach
\cite{Vegt:2008, Hughes:1988, Masud:1997, Ndri:2001, Zanotti:2015}. In
this paper we consider the space-time framework.

We describe the two-fluid model and its space-time EDG/HDG
discretization in, respectively, \cref{sec:problem,sec:ST_HDG}. The
discretization of the incompressible two-fluid model is based on the
space-time HDG method developed in \cite{Horvath:2019, Horvath:2020}
for the single-phase incompressible Navier--Stokes equations. This
discretization, like the ALE discretizations in \cite{Fu:2020,
  Neunteufel:2020}, is exactly mass conserving, even on moving
meshes. For the level-set equation, however, we do not use HDG. A
continuous representation of the zero level-set is required to fit the
mesh to the interface between the two fluids. The HDG method, however,
results in a discontinuous polynomial approximation of the level-set
function and so a smoothing technique would be required to obtain a
continuous approximation of the interface. It was shown in
\cite{Aizinger:2006}, however, that such smoothing may lead to
instabilities unless the discretization is modified by adding extra
stabilization terms. Instead, we propose to use a space-time EDG
discretization for the level-set equation. The EDG method directly
results in a \emph{continuous} approximation to the free-surface
resulting in a straightforward approach to update the mesh. In
\cref{sec:ST_HDG} we furthermore show that the space-time HDG method
for the incompressible two-fluid model is \emph{compatible} with the
space-time EDG method for the level-set equation, i.e., that the
space-time EDG method is able to preserve the constant solution. As
discussed in \cite{Dawson:2004}, for discontinuous Galerkin methods
compatibility is a stronger statement than local conservation of the
flow field. It was furthermore shown in \cite{Dawson:2004} that if a
method is not compatible it may produce erroneous solutions. In
\cref{sec:coupling_and_meshmovement} we describe the solution
algorithm while numerical simulations in \cref{sec:numerical_results}
demonstrate optimal rates of convergence and that our discretization
is energy-stable under mesh movement. We furthermore apply the
discretization to simulate sloshing in a tank and flow past a
submerged obstacle. Conclusions are drawn in \cref{sec:conclusions}.

\section{The space-time incompressible two-fluid flow model}
\label{sec:problem}

Let us consider a bounded domain $\Omega \subset \mathbb{R}^2$ and let
$I=(0,t_N]$ denote our time interval of interest. We then introduce a
space-time domain by
$\mathcal{E} := \Omega \times I \subset \mathbb{R}^3$. We will assume
that the space-time domain $\mathcal{E}$ is divided into two
non-overlapping polygonal regions, $\mathcal{E}_\ell$ and
$\mathcal{E}_g$, such that
$\mathcal{E} = \mathcal{E}_\ell \cup \mathcal{E}_g$. In what follows,
$\mathcal{E}_\ell$ and $\mathcal{E}_g$ represent, respectively, the
liquid and gas regions of the space-time domain. We denote the liquid
and gas regions at a particular time level $\tau$ by
$\Omega_\ell(\tau) := \cbr[0]{(x,t) \in \mathcal{E}_\ell\ :\ t=\tau}$
and $\Omega_g(\tau) := \cbr[0]{(x,t) \in \mathcal{E}_g\ :\
  t=\tau}$. Here $x=(x_1,x_2)$. Note that the spatial domains
$\Omega_\ell$ and $\Omega_g$ are \emph{time-dependent}.

The space-time interface between the liquid and gas regions is defined
as
\begin{equation}
  \label{eq:st_interface} 
  \mathcal{S} := \cbr[1]{ (x,t) \in \mathcal{E}\ :\ x_2 = \zeta(x_1,t) },
\end{equation}
where $\zeta(x_1, t)$ is the wave height. We will denote the interface
at time level $\tau$ by
$\Gamma_s(\tau) := \cbr[0]{(x,t) \in \mathcal{S}\ :\ t=\tau}$. A plot
of the domain at some time level $\tau$ is given in \cref{fig:domain}.

\begin{figure}[tbp]
  \begin{center}
    \begin{tikzpicture}[scale=0.8,
      important line/.style={thick},
      ]
      
      \draw[important line] (12,0) -- (12,8) ;
      \draw[important line] (12,8) -- (0,8) ;
      \draw[important line] (0,8) -- (0,0) ;

      \draw[color=blue, name path=A] (0,4) .. controls (3,2) and (4,5) .. (6,4);
      \draw[color=blue, name path=B] (6,4) .. controls (7,3) and (9,6) .. (12,4);

      \draw[name path=C] (0,0) .. controls (2,0.5) and (4,0.5) .. (6,0);
      \draw[name path=D] (6,0) .. controls (8,-0.5) and (10,-0.5) .. (12,0);

      \tikzfillbetween[of=A and C]{cyan, opacity=0.2};
      \tikzfillbetween[of=B and D]{cyan, opacity=0.2};

      \draw[dashed] (0,4) -- (12.8,4) ;
      \draw (13.8,4) node{$x_2 = 0$};
      \draw (10,5) node{$\Gamma_s(\tau)$};
      \draw (6,2) node{Liquid region $\Omega_\ell(\tau)$};
      \draw (6,6) node{Gas region $\Omega_g(\tau)$};
    \end{tikzpicture}    
  \end{center}
  \caption{A description of the two-fluid flow domain
    $\Omega \subset \mathbb{R}^2$ at time $t=\tau$.}
  \label{fig:domain}
\end{figure}
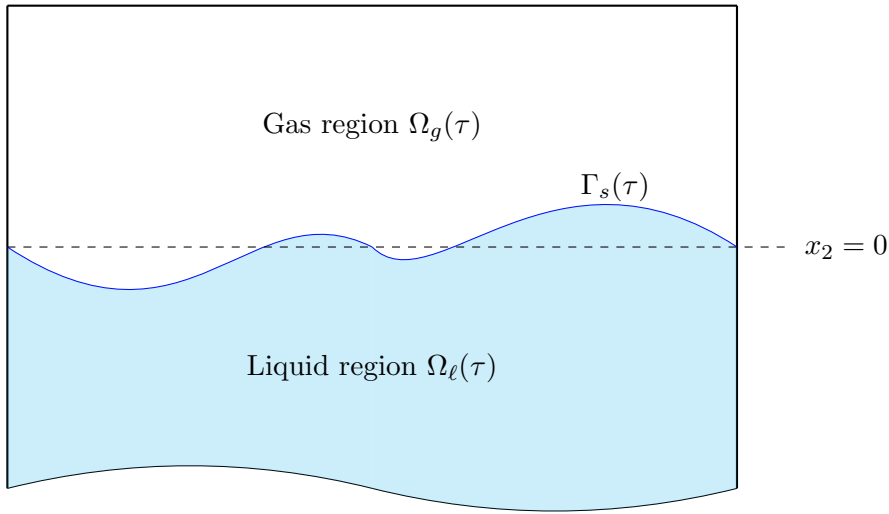

The position of the interface \cref{eq:st_interface} is not known a
priori. To find its position we introduce the level-set function
$\phi(x,t) = \zeta(x_1, t) - x_2$ and the Heaviside function $H(\phi)$
which is defined by
\begin{equation*}
  H(\phi) =
  \begin{cases}
    1 & \text{if } \phi > 0,
    \\
    0 & \text{if } \phi < 0.
  \end{cases}
\end{equation*}
We note that the interface corresponds to the zero level-set,
$\phi = 0$.

Let $k$ be a subscript to denote a liquid ($k=\ell$) or a gas ($k=g$)
property. Then given the dynamic viscosities $\mu_k \in \mathbb{R}^+$,
the constant densities $\rho_k \in \mathbb{R}^+$, and the constant
acceleration due to gravity $g$, the space-time formulation of the
incompressible two-fluid flow model for the velocity
$u : \mathcal{E} \to \mathbb{R}^2$ and pressure
$p : \mathcal{E} \to \mathbb{R}$ is given by
\begin{subequations}
  \begin{align}
    \label{eq:NS_mom}
    \rho \del[1]{\partial_t u + u \cdot \nabla u} + \nabla p - \nabla \cdot 2 \mu \varepsilon(u) &=  - \rho g e_2
    && \text{ in } \mathcal{E},
    \\
    \label{eq:NS_mass}
    \nabla \cdot u &= 0
    && \text{ in } \mathcal{E},
    \\
    \label{eq:NS_kin}
    \partial_t\phi + u \cdot \nabla \phi &= 0
    && \text{ in } \mathcal{E},
  \end{align}
  \label{eq:twofluidNS}
\end{subequations}
where $\varepsilon(u) := (\nabla u + (\nabla u)^T) / 2$ is the
symmetric gradient, $e_2$ is the unit vector in the $x_2$-direction
and
\begin{subequations}
  \label{eq:uprhomu}
  \begin{align}
    \label{eq:uprhomu_a}
    u &= u_g + (u_\ell - u_g)H(\phi), & p &= p_g + (p_\ell - p_g)H(\phi), \\
    \label{eq:uprhomu_b}
    \rho &= \rho_g + (\rho_\ell - \rho_g)H(\phi), & \mu &= \mu_g + (\mu_\ell - \mu_g)H(\phi).
  \end{align}
\end{subequations}

Let $\Omega_0 := \cbr[0]{(x,t) \in \partial\mathcal{E}\ :\ t=0}$ and
similarly
$\Omega_N := \cbr[0]{(x,t) \in \partial\mathcal{E}\ :\ t=t_N}$. Then
the boundary of the space-time domain $\mathcal{E}$ is partitioned
such that
$\partial\mathcal{E} = \partial\mathcal{E}^D \cup
\partial\mathcal{E}^N \cup \Omega_0 \cup \Omega_N$, where there is no
overlap between the four sets. Here, $\partial\mathcal{E}^D$ and
$\partial\mathcal{E}^N$ denote, respectively, the Dirichlet and
Neumann parts of the space-time boundary. The space-time outward unit
normal vector to $\partial\mathcal{E}$ is denoted by $(n_t, n)$, with
$n_t \in \mathbb{R}$ the temporal component and $n \in \mathbb{R}^2$
the spatial component. We define the inflow boundary
$\partial\mathcal{E}^-$ as the portion of $\partial\mathcal{E}^N$ on
which $n_t + u\cdot n < 0$. The outflow boundary is then defined as
$\partial\mathcal{E}^+ := \partial\mathcal{E}^N \backslash
\partial\mathcal{E}^-$. We prescribe the following boundary and
initial conditions:
\begin{subequations}
  \label{eq:bcic}
  \begin{align}
    \label{eq:bcic_a}
    u &= 0 && \text{on } \partial\mathcal{E}^D,
    \\
    \label{eq:bcic_b}
    \rho\sbr[0]{n_t + u\cdot n - \max\del[0]{n_t + u\cdot n,0}}u + \del[0]{p\mathbb{I} - 2\mu\varepsilon(u)}n
      &=f && \text{on } \partial\mathcal{E}^N,
    \\
    u(x,0) &= u_0(x) && \text{in } \Omega_0,
    \\
    \label{eq:bcic_d}
    -(n_t + u \cdot n)\phi &= r && \text{on } \partial\mathcal{E}^-,
    \\
    \label{eq:bcic_e}
    \phi(x,0) &= \phi_0(x) && \text{in } \Omega_0,
  \end{align}
\end{subequations}
where the boundary data $f:\partial\mathcal{E}^N \to \mathbb{R}^d$ and
$r:\partial\mathcal{E}^- \to \mathbb{R}$, and the divergence-free
initial condition $u_0 : \Omega_0 \to \mathbb{R}^d$ are
given. Furthermore, $\phi_0(x) := \zeta_0(x_1) - x_2$ with
$\zeta_0(x_1)$ the given initial wave height.

We remark that the two-fluid problem assumes continuity of the
velocity and of the diffusive flux across the interface $\mathcal{S}$,
i.e.,
\begin{subequations}
  \label{eq:interfaceassump}
  \begin{align}
    \label{eq:continuity_u}
    u_\ell &= u_g && \text{on } \mathcal{S},
    \\
    \label{eq:continuity_stress}
    \del[0]{p_\ell\mathbb{I} - 2\mu_\ell\varepsilon(u_\ell)}n_s
        &=
          \del[0]{p_g\mathbb{I} - 2\mu_g\varepsilon(u_g)}n_s && \text{on } \mathcal{S},
  \end{align}
\end{subequations}
where $n_s$ is the normal vector on the interface $\Gamma_s(t)$
pointing outwards from $\Omega_\ell(t)$.

\section{The space-time hybridizable/embedded discontinuous Galerkin
  method}
\label{sec:ST_HDG}

In this section we will introduce the space-time discretization for
the two-fluid problem \cref{eq:twofluidNS,eq:bcic}. In particular we
will introduce a space-time hybridizable discontinuous Galerkin (HDG)
for the momentum and mass \cref{eq:NS_mom,eq:NS_mass} coupled to a
space-time embedded discontinuous Galerkin (EDG) discretization of the
level-set equation \cref{eq:NS_kin}.

\subsection{Space-time notation}
\label{ss:ST_notation}

We will use notation similar to, for example, \cite{Horvath:2020}. For
this we first partition the time interval $I$ into time levels
$0 = t_0 < t_1 < t_2 < \hdots < t_N$ and denote the $n^{\text{th}}$
time interval by $I_n = (t_n,t_{n+1})$. A `time-step' is defined as
the length of a time interval, i.e., $\Delta t = t_{n+1}-t_n$. Next,
we define the $n^{\text{th}}$ space-time slab as
$\mathcal{E}^n := \mathcal{E} \cap \del{I_n \times \mathbb{R}^2}$, we
define the spatial domain at time level $t_n$ as
$\Omega_n := \cbr[0]{(x,t) \in\mathcal{E} : t = t_n}$, and note that
the boundary of a space-time slab, denoted by $\partial\mathcal{E}^n$,
can be divided into $\Omega_n$, $\Omega_{n+1}$, and
$\mathcal{Q}_{\mathcal{E}}^n := \partial\mathcal{E}^n \backslash
\del{\Omega_{n+1} \cup \Omega_n}$.

Following \cite{Horvath:2019, Horvath:2020} we consider a tetrahedral
space-time mesh which is constructed as follows. First, the triangular
spatial mesh of $\Omega_n$ is extruded to the new time level $t_{n+1}$
according to the domain deformation prescribed by the wave height.
(Note that the wave height is not known a priori and so, as we will
describe in \cref{sec:coupling_and_meshmovement}, an iterative
procedure is used to find the final approximation to the domain at
time level $t_{n+1}$.) Each element of this prismatic space-time mesh
is then subdivided into three tetrahedrons. We denote the space-time
triangulation in the space-time slab $\mathcal{E}^n$ by
$\mathcal{T}^n := \cbr[0]{\mathcal{K}}$. The space-time elements
$\mathcal{K} \in \mathcal{T}^n$ that lie in the liquid region of the
space-time slab form the triangulation
$\mathcal{T}^n_\ell$. $\mathcal{T}^n_g$ is defined similarly but for the
gas region. The triangulation of the whole space-time domain
$\mathcal{E}$ is denoted by $\mathcal{T} :=
\cup_{n}\mathcal{T}^n$.

The boundary of a space-time tetrahedron
$\mathcal{K}_j \in \mathcal{T}^n$ is denoted by
$\partial\mathcal{K}_j$ and the outward unit space-time normal vector
on the boundary of $\mathcal{K}_j \in \mathcal{T}^n$ is given by
$(n_t^{\mathcal{K}_j},n^{\mathcal{K}_j})$. The boundary
$\partial\mathcal{K}_j$ consists of at most one face that belongs to a
time level (on which $\envert[0]{n_t^{\mathcal{K}_j}} = 1$). We denote
this face by $K_j^n$ if $n_t^{\mathcal{K}_j} = -1$ and by $K_j^{n+1}$
if $n_t^{\mathcal{K}_j} = 1$. The remaining faces of
$\partial\mathcal{K}_j$ are denoted by
$\mathcal{Q}_{\mathcal{K}_j}^n := \partial\mathcal{K}_j \backslash
K_j^n$ or
$\mathcal{Q}_{\mathcal{K}_j}^n := \partial\mathcal{K}_j \backslash
K_j^{n+1}$. Note that the space-time normal on
$\mathcal{Q}_{\mathcal{K}_j}^n$ depends on the grid velocity
$v_g \in \mathbb{R}^2$ as follows:
$n_t^{\mathcal{K}_j} = -v_g \cdot n$. This relation allows to switch
between a space-time formulation and an arbitrary Lagrangian Eulerian
(ALE) formulation (see, e.g., \cite{Vegt:2002}). In the remainder of
this article, to simplify notation, we drop the sub- and superscript
when referring to the space-time normal vector and the space-time cell
wherever no confusion will occur.

In a space-time slab $\mathcal{E}^n$, the set of all faces for which
$\envert[0]{n_t} \neq 1$ is denoted by $\mathcal{F}^n$ while the union
of these faces is denoted by $\Gamma^n$. By
$\mathcal{F}_{\mathcal{S}}^n$ we denote all the faces in
$\mathcal{F}^n$ that lie on the interface $\mathcal{S}$. Similarly,
the set of all faces in $\mathcal{F}^n$ that lie on the boundary of
the space-time domain is denoted by $\mathcal{F}^n_B$. The remaining
set of (interior) faces is denoted by $\mathcal{F}_I^n$. Then
$\mathcal{F}^n = \mathcal{F}_I^n \cup \mathcal{F}_{\mathcal{S}}^n \cup
\mathcal{F}_B^n$. Furthermore, we will denote the set of faces that
lie on a Neumann boundary,
$\partial\mathcal{E}^N \cap \partial\mathcal{E}^n$, by
$\mathcal{F}_N^n$.

On each space-time slab $\mathcal{E}^n$ we consider the following
discontinuous finite element spaces on $\mathcal{T}^n$:
\begin{subequations}
  \begin{align}
    V_h^n
    & := \cbr[1]{v_h \in \sbr[0]{L^2(\mathcal{T}^n)}^2\, : 
      v_h\vert_{\mathcal{K}}
      \in \sbr[1]{P^{k}(\mathcal{K})}^2,
      \,\forall \mathcal{K}\in\mathcal{T}^n},
    \\
    Q_h^n
    &:= 
      \cbr[1]{q_h \in L^2(\mathcal{T}^n) : q_h \vert_{\mathcal{K}} 
      \in P^{k-1}(\mathcal{K}), 
      \,\forall \mathcal{K} \in \mathcal{T}^n},
    \\
    M_h^n
    &:= 
      \cbr[1]{m_h \in L^2(\mathcal{T}^n) : m_h \vert_{\mathcal{K}} 
      \in P^{k}(\mathcal{K}), 
      \,\forall \mathcal{K} \in \mathcal{T}^n},
  \end{align}
\end{subequations}
where $P^l(D)$ denotes the space of polynomials of degree $l$ on a
domain $D$. Additionally, we consider the following facet finite
element spaces:
\begin{subequations}
  \begin{align}
    \bar{V}_h^n
    & := \cbr[1]{\bar{v}_h \in \sbr[0]{L^2(\mathcal{F}^n)}^2\, : 
      \bar{v}_h\vert_{\mathcal{F}} 
      \in \sbr[0]{P^{k}(\mathcal{F})}^2,
      \,\forall \mathcal{F}\in\mathcal{F}^n,
      \bar{v}_h = 0 \text{ on } \partial\mathcal{E}^D \cap \partial\mathcal{E}^n},
    \\
    \bar{Q}_h^n
    &:= 
      \cbr[1]{\bar{q}_h \in L^2(\mathcal{F}^n) :
      \bar{q}_h \vert_{\mathcal{F}} 
      \in P^{k}(\mathcal{F}), 
      \,\forall \mathcal{F} \in \mathcal{F}^n},
    \\
    \label{eq:barMhn}
    \bar{M}_h^n
    &:= 
      \cbr[1]{\bar{m}_h \in L^2(\mathcal{F}^n) : \bar{m}_h \vert_{\mathcal{F}} 
      \in P^{k}(\mathcal{F}), 
      \,\forall \mathcal{F} \in \mathcal{F}^n} \cap C(\Gamma^n).
  \end{align}
\end{subequations}
Note that the facet velocity field in $\bar{V}_h^n$ and facet pressure
field in $\bar{Q}_h^n$ are discontinuous while the facet level-set
field in $\bar{M}_h^n$ is continuous. To simplify the notation, we
introduce $X_h^{v,n} := V_h^n \times \bar{V}_h^n$,
$X_h^{q,n} = Q_h^n \times \bar{Q}_h^n$,
$X_h^n = X_h^{v,n} \times X_h^{q,n}$, and
$X_h^{m,n} = M_h^n \times \bar{M}_h^n$. Function pairs in $X_h^{v,n}$,
$X_h^{q,n}$, and $X_h^{m,n}$ will be denoted by boldface.

\subsection{Discretization of the momentum and mass equations}
\label{ss:NS_disc}

An exactly mass conserving space-time HDG discretization for the
single-phase Navier--Stokes equations was introduced in
\cite{Horvath:2019, Horvath:2020}. Here we modify this discretization
to take into account that the density and viscosity may be
discontinuous across elements (see \cref{eq:uprhomu}).

Since the density and viscosity is constant on each element
$\mathcal{K} \in \mathcal{T}^n$ we define
$\rho_{\mathcal{K}} = \rho_g$ if $\mathcal{K} \in \mathcal{T}_g^n$ and
$\rho_{\mathcal{K}} = \rho_\ell$ if
$\mathcal{K} \in \mathcal{T}_\ell^n$ (and $\mu_{\mathcal{K}}$ is
defined similarly). The space-time HDG discretization for the momentum
and mass equations \cref{eq:NS_mom,eq:NS_mass} is then given by: find
$(\boldsymbol{u}_h, \boldsymbol{p}_h) \in X_h^{n}$ such that
\begin{multline}
  \label{eq:NSdisc}
  B_{conv}^n(\boldsymbol{u}_h, \boldsymbol{u}_h, \boldsymbol{v}_h)
  + B_{dif}^n(\boldsymbol{u}_h, \boldsymbol{v}_h)
  + B_{pu}^n(\boldsymbol{p}_h, \boldsymbol{v}_h)
  - B_{pu}^n(\boldsymbol{q}_h, \boldsymbol{u}_h) 
  \\
  = -\sum_{\mathcal{K} \in \mathcal{T}^n}\int_{\mathcal{K}}\rho_{\mathcal{K}} g e_2 \cdot v_h \dif x\dif t
  - \sum_{\mathcal{F} \in 	\mathcal{F}_N^n} \int_{\mathcal{F}} f \cdot \bar{v}_h \dif s
  + \int_{\Omega_n} \rho_{\mathcal{K}} u_h^-\cdot v_h \dif x,  
\end{multline}
for all $(\boldsymbol{v}_h, \boldsymbol{q}_h) \in X_h^{n}$, where
$u_h^- = \lim_{\epsilon\rightarrow 0} u_h(x,t_n - \epsilon)$ for
$n > 0$. When $n=0$ then $u_h^-$ is the projection of the initial
condition $u_0$ into $V_h^0 \cap H(\text{div})$ such that it is
exactly divergence-free. Let $\mathcal{K}^-$ and $\mathcal{K}$ denote
two elements that share a facet $\mathcal{F}$ on the boundary
$\mathcal{Q}_{\mathcal{K}}$ and let
$\widehat{\rho}_{\mathcal{K}} = \del[0]{\rho_{\mathcal{K}} +
  \rho_{\mathcal{K}^-}}/2$. We then define the convective trilinear
form $B_{conv}^n$ as
\begin{align}
  \label{eq:t_h_def}
  B_{conv}^n(\boldsymbol{w};\boldsymbol{u},\boldsymbol{v})
  :=& -\sum_{\mathcal{K} \in \mathcal{T}^n}\int_{\mathcal{K}}
      \rho_{\mathcal{K}}\del[0]{u\cdot \partial_t v + u\otimes w : \nabla v}\dif x\dif t
      + \sum_{\mathcal{K} \in \mathcal{T}^n} \int_{K^{n+1}} \rho_{\mathcal{K}} u \cdot v \dif x
  \\
  \nonumber
    &
      + \sum_{\mathcal{K} \in \mathcal{T}^n}\int_{\mathcal{Q}_{\mathcal{K}}}
      \widehat{\rho}_{\mathcal{K}}\del[0]{n_t + w\cdot n}\del[0]{u + \lambda\del[0]{\bar{u} - u}}\cdot \del[0]{v - \bar{v}}
      \dif s
  \\
  \nonumber
  & + \int_{\partial\mathcal{E}^+}\rho_{\mathcal{K}}\del[0]{n_t + \bar{w} \cdot n}\bar{u} \cdot \bar{v} \dif s
    + \sum_{\mathcal{K} \in \mathcal{T}^n}\int_{\mathcal{Q}_{\mathcal{K}}}
    \del[0]{\rho_{\mathcal{K}} - \widehat{\rho}_{\mathcal{K}}}\del[0]{n_t + w\cdot n} u \cdot v\dif s,
\end{align}
where $\lambda = 1$ if $n_t + w\cdot n < 0$ and $\lambda = 0$
otherwise. This corresponds to an upwind numerical flux on the element
boundaries. Note that if $\rho_{\ell} = \rho_g$, the last integral on
the right hand side of \cref{eq:t_h_def} disappears and the trilinear
form reduces exactly to the trilinear form introduced in
\cite{Horvath:2019} for single-phase flow. The last integral on the
right hand side of \cref{eq:t_h_def} is due to
\cref{eq:continuity_stress} not including any inertial effects; only
continuity of the diffusive flux is imposed across the interface.

The diffusive $B_{dif}^n$ and velocity-pressure coupling $B_{pu}^n$
bilinear forms are the same as in the single-phase flow problem, but
with a discontinuous viscosity. They are given by
\begin{subequations}
  \begin{align}
    \label{eq:a_h_def}
    B_{dif}^n(\boldsymbol{u}, \boldsymbol{v})
    := & \sum_{\mathcal{K} \in \mathcal{T}^n}\int_{\mathcal{K}} 2 \mu_{\mathcal{K}} \varepsilon(u) : \varepsilon(v) \dif x\dif t 
         + \sum_{\mathcal{K} \in \mathcal{T}^n}\int_{\mathcal{Q}_{\mathcal{K}}}\frac{2\mu_{\mathcal{K}}\alpha}{ h_{\mathcal{K}}}\del[0]{u - \bar{u}}\cdot \del[0]{v - \bar{v}}\dif s
    \\
    \nonumber
       &- \sum_{\mathcal{K} \in \mathcal{T}^n}\int_{\mathcal{Q}_{\mathcal{K}}}
         2\mu_{\mathcal{K}}\sbr[0]{\del[0]{u - \bar{u}} \cdot \varepsilon(v)n + \varepsilon(u)n \cdot \del[0]{v - \bar{v}}} \dif s,         
    \\
    \label{eq:b_h_def}
    B_{pu}^n(\boldsymbol{p},\boldsymbol{v})
    := & -\sum_{\mathcal{K} \in \mathcal{T}^n}\int_{\mathcal{K}} p\nabla\cdot v\dif x\dif t
         + \sum_{\mathcal{K} \in \mathcal{T}^n}\int_{\mathcal{Q}_{\mathcal{K}}} \del[0]{v - \bar{v}} \cdot n \bar{p} \dif s.
  \end{align}
\end{subequations}
Here $\alpha > 0$ is the interior-penalty parameter.

To find the solution of the nonlinear discrete problem
\cref{eq:NSdisc}, we use a Picard iteration scheme: in every
space-time slab, given $(\boldsymbol{u}_h^{k}, \boldsymbol{p}_h^{k})$
we seek a solution $(\boldsymbol{u}_h^{k+1}, \boldsymbol{p}_h^{k+1})$
to the linear discrete problem
\begin{multline}
  \label{eq:PicardNSdisc}
  B_{conv}^n(\boldsymbol{u}_h^{k}; \boldsymbol{u}_h^{k+1}, \boldsymbol{v}_h)
  + B_{dif}^n(\boldsymbol{u}_h^{k+1}, \boldsymbol{v}_h)
  + B_{pu}^n(\boldsymbol{p}_h^{k+1}, \boldsymbol{v}_h)
  - B_{pu}^n(\boldsymbol{q}_h, \boldsymbol{u}_h^{k+1}) 
  \\
  = -\sum_{\mathcal{K} \in \mathcal{T}^n}\int_{\mathcal{K}}\rho_{\mathcal{K}} g e_2 \cdot v_h \dif x\dif t
  - \sum_{\mathcal{F} \in \mathcal{F}_N^n} \int_{\mathcal{F}} f \cdot \bar{v}_h \dif s
  + \int_{\Omega_n} \rho_{\mathcal{K}} u_h^-\cdot v_h \dif x,  
\end{multline}
for all $(\boldsymbol{v}_h, \boldsymbol{q}_h) \in X_h^{n}$ and for
$k=0,1,2,\hdots$ until the following stopping criterium is met:
\begin{equation}
  \max\cbr{\frac{\norm[0]{u_h^{k} - u_h^{k-1}}_{\infty}}{\norm[0]{u_h^{k} - u_h^{0}}_{\infty}}, \frac{\norm[0]{p_h^{k} - p_h^{k-1}}_{\infty}}{\norm[0]{p_h^{k} - p_h^{0}}_{\infty}}} < \varepsilon_{u,p},
\end{equation}
where $\varepsilon_{u,p}$ is a user given parameter. We then set
$(\boldsymbol{u}_h, \boldsymbol{p}_h) = (\boldsymbol{u}_h^{k+1},
\boldsymbol{p}_h^{k+1})$.

\subsection{Discretization of the level-set equation}
\label{ss:levelset_disc}

The space-time EDG discretization for the level-set equation
\cref{eq:NS_kin} is a space-time extension of the EDG discretization
for the advection equation \cite{Wells:2011}: In each space-time slab
$\mathcal{E}^n$, for $n = 0,1,\ldots, N-1$, given $u$, find
$\boldsymbol{\Phi}_h \in X_h^{m,n}$ such that
\begin{equation}
  \label{eq:levelset_disc}
  C_{ls}(\boldsymbol{\Phi}_h, \boldsymbol{m}_h; u)
  = \sum_{\mathcal{K} \in \mathcal{T}^n}\int_{K^n}\phi_h^- m_h \dif x
  + \int_{\partial\mathcal{E}^-}r \bar{m}_h\dif s
  \qquad \forall \boldsymbol{m}_h \in X_h^{m,n},
\end{equation}
where $\phi_h^- = \lim_{\epsilon\rightarrow 0}\phi_h(x,t_n-\epsilon)$
for $n>0$. When $n=0$ $\phi_h^-$ is the projection of the initial
condition $\phi_0$ into $M_h^0$. The bilinear form $C_{ls}$ is given
by
\begin{multline}
  \label{eq:definition_c}
  C_{ls}(\boldsymbol{\Phi}, \boldsymbol{m}; u)
  = -\sum_{\mathcal{K} \in \mathcal{T}^n} \int_{\mathcal{K}} \del[0]{\phi\partial_t m + \phi u\cdot\nabla m}\dif x\dif t
  + \sum_{\mathcal{K} \in \mathcal{T}^n}\int_{K^{n+1}} \phi m \dif x 
  \\
  + \sum_{\mathcal{K} \in \mathcal{T}^n}\int_{\mathcal{Q}_{\mathcal{K}}}
  \del[0]{n_t + u\cdot n}\del{\phi + \lambda\del[0]{\bar{\phi} - \phi}}\del[0]{m - \bar{m}} \dif s
  + \int_{\partial\mathcal{E}^+}(n_t + u \cdot n) \bar{\phi} \bar{m}\dif s.
\end{multline}

\subsection{Properties of the discretization}

In this section we discuss properties of the space-time HDG/EDG
discretization, \cref{eq:NSdisc,eq:levelset_disc}, of the two-fluid
flow model.

First, we remark that the discretization conserves mass exactly. The
proof of this result is identical to \cite[Prop. 1]{Horvath:2020}:
that $u_h$ is exactly divergence free ($\nabla\cdot u_h = 0$) on the
elements follows by taking $\boldsymbol{v}_h = 0$, $\bar{q}_h=0$ and
$q_h = \nabla \cdot u_h$ in \cref{eq:NSdisc} while
$H(\text{div})$-conformity of $u_h$ (i.e., $\jump{u_h\cdot n}=0$ on
interior facets and $u_h\cdot n = \bar{u}_h\cdot n$ on boundary
facets) follows by setting $\boldsymbol{v}_h=0$, $q_h=0$, and
$\bar{q}_h = \jump{(u_h - \bar{u}_h)\cdot n}$ in \cref{eq:NSdisc}.

It was shown in \cite{Dawson:2004} for different flow/transport
discretizations that loss of accuracy and/or loss of global
conservation may occur if a discretization is not compatible. (Note
that compatibility for discontinuous Galerkin methods is a stronger
statement than local conservation of the flow field
\cite{Dawson:2004}.) The next result shows that since $u_h$ is exactly
divergence free, the space-time HDG discretization of the two-fluid
model \cref{eq:NSdisc} and the space-time EDG discretization of the
level-set equation \cref{eq:levelset_disc} are compatible.

\begin{proposition}[Compatibility]
  \label{prop:compatibility}
  If $u_h \in V_h^n$ is the velocity solution to the space-time HDG
  discretization \cref{eq:NSdisc}, then the space-time EDG
  discretization \cref{eq:levelset_disc} is: (i) globally conservative;
  and (ii) able to preserve the constant solution.
\end{proposition}
\begin{proof}
  We first note that global conservation of the EDG discretization was
  shown in \cite{Wells:2011}. To show that the space-time EDG
  discretization \cref{eq:levelset_disc} is able to preserve the
  constant solution we present a space-time extension of the
  discussion in \cite[Section 3.4]{Cesmelioglu:2021}. For this, let
  the boundary and initial conditions in \cref{eq:bcic_d,eq:bcic_e} be
  given by, respectively, $\phi_0(x) = \psi$ and
  $r = -(n_t+u_h\cdot n)\psi$, where $\psi$ is a constant. Consider now
  the first space-time slab $\mathcal{E}^0$. The constant
  $\boldsymbol{\Phi}_h=(\psi,\psi)$ is preserved by
  \cref{eq:levelset_disc} if and only if
  \begin{equation}
    \label{eq:levelset_disc_comp}
    C_{ls}((\psi,\psi), \boldsymbol{m}_h; u_h)
    = \sum_{\mathcal{K} \in \mathcal{T}^0}\int_{K^n} \psi m_h \dif x
    - \int_{\partial\mathcal{E}^-}(n_t+u_h\cdot n)\psi \bar{m}_h\dif s
    \qquad \forall \boldsymbol{m}_h \in X_h^{m,n}.
  \end{equation}
  If $\psi = 0$ it is clear that \cref{eq:levelset_disc_comp}
  holds. Consider therefore the case that $\psi \neq 0$. Writing out
  the left hand side, dividing both sides by $\psi$, and integrating
  by parts in time, we find that \cref{eq:levelset_disc_comp} is
  equivalent to
  \begin{multline}
    \label{eq:definition_comp2}  
    - \sum_{\mathcal{K} \in \mathcal{T}^0} \int_{\mathcal{K}} u_h\cdot\nabla m_h\dif x\dif t
    - \sum_{\mathcal{K} \in \mathcal{T}^0} \int_{\mathcal{Q}_{\mathcal{K}}}n_tm_h \dif s
    + \sum_{\mathcal{K} \in \mathcal{T}^0}\int_{\mathcal{Q}_{\mathcal{K}}}
    \del[0]{n_t + u_h\cdot n}\del[0]{m_h - \bar{m}_h} \dif s
    \\
    + \int_{\partial\mathcal{E}^+}(n_t + u_h \cdot n) \bar{m}_h\dif s
    = 0,
  \end{multline}
  for all $\boldsymbol{m}_h \in X_h^{m,n}$. Using that
  $\nabla\cdot(u_h m_h) = u_h\cdot\nabla m_h + m_h\nabla\cdot u_h$ on
  each element $\mathcal{K}$, integration by parts in space,
  single-valuedness of $\bar{m}_h$ and $u_h\cdot n$ on interior facets
  (by the $H(\text{div})$-conformity of $u_h$), we find that the
  constant $\boldsymbol{\Phi}_h=(\psi,\psi)$ is preserved by
  \cref{eq:levelset_disc} if and only if
  \begin{equation}
    \label{eq:definition_comp3}  
    \sum_{\mathcal{K} \in \mathcal{T}^0} \int_{\mathcal{K}} m_h \nabla\cdot u_h \dif x\dif t
    = 0.
  \end{equation}
  Since the velocity solution to \cref{eq:NSdisc} is exactly
  divergence free on the elements, the result follows.
\end{proof}

We emphasize that compatibility in \Cref{prop:compatibility} between
the space-time HDG discretization \cref{eq:NSdisc} and the space-time
EDG discretization \cref{eq:levelset_disc} is a direct consequence of
the exact mass conservation property of the space-time HDG method for
the incompressible two-fluid flow equations. A discretization that is
not exactly mass conserving may not be compatible with
\cref{eq:levelset_disc}.

We end this section by showing consistency of the space-time HDG/EDG
discretization.

\begin{proposition}[Consistency]
  Let $u(x,t)$, $p(x,t)$, and $\phi(x,t)$ be the smooth solution to
  the two-fluid model, \cref{eq:twofluidNS,eq:bcic}. Let
  $\boldsymbol{u}=(u,u)$, $\boldsymbol{p}=(p,p)$, and
  $\boldsymbol{\Phi}=(\phi,\phi)$. Then
  \begin{multline}
    \label{eq:NSdisc-consistency}
    B_{conv}^n(\boldsymbol{u}, \boldsymbol{u}, \boldsymbol{v}_h)
    + B_{dif}^n(\boldsymbol{u}, \boldsymbol{v}_h)
    + B_{pu}^n(\boldsymbol{p}, \boldsymbol{v}_h)
    - B_{pu}^n(\boldsymbol{q}, \boldsymbol{u}_h) 
    \\
    = -\sum_{\mathcal{K} \in \mathcal{T}^n}\int_{\mathcal{K}}\rho_{\mathcal{K}} g e_2 \cdot v_h \dif x\dif t
    - \sum_{\mathcal{F} \in \mathcal{F}_N^n} \int_{\mathcal{F}} f \cdot \bar{v}_h \dif s
    + \int_{\Omega_n} \rho_{\mathcal{K}} u \cdot v_h \dif x \qquad \forall (\boldsymbol{v}_h, \boldsymbol{q}_h) \in X_h^{n},  
  \end{multline}
  and
  \begin{equation}
    \label{eq:levelset_disc_consistency}
    C_{ls}(\boldsymbol{\Phi}, \boldsymbol{m}_h; u)
    = \sum_{\mathcal{K} \in \mathcal{T}^n}\int_{K^n}\phi m_h \dif x
    + \int_{\partial\mathcal{E}^-}r \bar{m}_h\dif s
    \qquad \forall \boldsymbol{m}_h \in X_h^{m,n}.
  \end{equation}
\end{proposition}
\begin{proof}
  We first show \cref{eq:NSdisc-consistency}. By definition
  \cref{eq:t_h_def}, integration by parts, and using that
  $\widehat{\rho}_{\mathcal{K}}$, $u$ and $\bar{v}_h$ are
  single-valued on faces, we find
  \begin{multline}
    \label{eq:consistency_t}
    B_{conv}^n(\boldsymbol{u}; \boldsymbol{u}, \boldsymbol{v}_h)
    = \sum_{\mathcal{K} \in \mathcal{T}^n}\int_{\mathcal{K}}
    \rho_{\mathcal{K}}\del[0]{\partial_t u + u\cdot\nabla u} \cdot v_h\dif x\dif t
    + \sum_{\mathcal{K} \in \mathcal{T}^n} \int_{K^{n}} \rho_{\mathcal{K}} u \cdot v_h \dif x
    \\
    - \int_{\partial\mathcal{E}^-}\rho_{\mathcal{K}}\del[0]{n_t + u \cdot n} u \cdot \bar{v}_h \dif s.
  \end{multline}
  Similarly, by definition \cref{eq:a_h_def} and integration by parts,
  \begin{equation}
    \label{eq:consistency_a}
    B_{dif}^n(\boldsymbol{u}, \boldsymbol{v}_h)
    = -\sum_{\mathcal{K} \in \mathcal{T}^n}\int_{\mathcal{K}}  \nabla \cdot (2\mu_{\mathcal{K}}\varepsilon(u)) \cdot v_h \dif x \dif t
    + \sum_{\mathcal{K} \in \mathcal{T}^n}\int_{\mathcal{Q}_{\mathcal{K}}} 2\mu_{\mathcal{K}} \varepsilon(u)n \cdot \overline{v}_h \dif s,
  \end{equation}
  and by definition \cref{eq:b_h_def}, integration by parts, and using
  that $\nabla \cdot u = 0$,
  \begin{equation}
    \label{eq:consistency_b}
    B_{pu}^n(\boldsymbol{p}, \boldsymbol{v}_h)
    - B_{pu}^n(\boldsymbol{q}_h, \boldsymbol{u})
    = \sum_{\mathcal{K}\in\mathcal{T}^n} \int_{\mathcal{K}} \nabla p \cdot v_h \dif x\dif t 
    - \sum_{\mathcal{K} \in \mathcal{T}^n}\int_{\mathcal{Q}_{\mathcal{K}}} \bar{v}_h\cdot n p \dif s.
  \end{equation}
  Combining \cref{eq:consistency_t,eq:consistency_a,eq:consistency_b},
  we obtain
  \begin{align}
    B_{conv}^n(\boldsymbol{u}; \boldsymbol{u}, \boldsymbol{v}_h)
    &+ B_{dif}^n(\boldsymbol{u}, \boldsymbol{v}_h)
    + B_{pu}^n(\boldsymbol{p}, \boldsymbol{v}_h)
      - B_{pu}^n(\boldsymbol{q}_h, \boldsymbol{u})
    \\ \nonumber
    =&
    \sum_{\mathcal{K} \in \mathcal{T}^n}\int_{\mathcal{K}}
       \rho_{\mathcal{K}}\sbr[1]{\del[0]{\partial_t u + u\cdot\nabla u} + \nabla p - \nabla \cdot (2\mu_{\mathcal{K}}\varepsilon(u)) }\cdot v_h\dif x\dif t
       + \sum_{\mathcal{K} \in \mathcal{T}^n} \int_{K^{n}} \rho_{\mathcal{K}} u \cdot v_h \dif x
    \\ \nonumber
    & - \sum_{\mathcal{K} \in \mathcal{T}^n}\int_{\mathcal{Q}_{\mathcal{K}}} \del{p\mathbb{I} - 2\mu_{\mathcal{K}}\varepsilon(u)}n \cdot \bar{v}_h \dif s
      - \int_{\partial\mathcal{E}^-}\rho_{\mathcal{K}}\del[0]{n_t + u \cdot n} u \cdot \bar{v}_h \dif s.
  \end{align}
  The last two terms may be combined using that $\bar{v}_h = 0$ on
  $\partial\mathcal{E}^D$ and the single-valuedness of
  $\del[0]{p\mathbb{I} - 2\mu_{\mathcal{K}}\varepsilon(u)}n$ on element
  boundaries:
  \begin{align}
    B_{conv}^n(\boldsymbol{u}; \boldsymbol{u}, \boldsymbol{v}_h)
    &+ B_{dif}^n(\boldsymbol{u}, \boldsymbol{v}_h)
    + B_{ls}^n(\boldsymbol{p}, \boldsymbol{v}_h)
      - B_{ls}^n(\boldsymbol{q}_h, \boldsymbol{u})
    \\ \nonumber
    =&
    \sum_{\mathcal{K} \in \mathcal{T}^n}\int_{\mathcal{K}}
       \rho_{\mathcal{K}}\sbr[1]{\del[0]{\partial_t u + u\cdot\nabla u} + \nabla p - \nabla\cdot(2\mu_{\mathcal{K}}\varepsilon(u))}\cdot v_h\dif x\dif t
       + \int_{\Omega_n} \rho_{\mathcal{K}} u \cdot v_h \dif x
    \\ \nonumber
    & - \int_{\partial\mathcal{E}^N}\del{\sbr[0]{n_t + u\cdot n - \max\del[0]{n_t + u\cdot n,0}}u + \del[0]{p\mathbb{I} - 2\mu\varepsilon(u)}n} \cdot \bar{v}_h \dif s.
  \end{align}
  \Cref{eq:NSdisc-consistency} follows using \cref{eq:NS_mom} and
  \cref{eq:bcic_b}.

  We next show \cref{eq:levelset_disc_consistency}. By definition
  \cref{eq:definition_c}, integration by parts, and using that $\phi$,
  $\bar{m}_h$, and $u$ are single-valued on faces, we find
  \begin{equation}
    C_{ls}(\boldsymbol{\Phi}, \boldsymbol{m}_h; u)
    = \sum_{\mathcal{K} \in \mathcal{T}^n} \int_{\mathcal{K}} \del[0]{\partial_t\phi + u \cdot \nabla \phi} m_h \dif x\dif t
    + \sum_{\mathcal{K} \in \mathcal{T}^n}\int_{K^{n}} \phi m_h \dif x
    - \int_{\partial\mathcal{E}^-}\del[0]{n_t + u\cdot n}\phi \bar{m}_h \dif s.
  \end{equation}
  \Cref{eq:levelset_disc_consistency} follows using
  \cref{eq:NS_kin,eq:bcic_d}.
\end{proof}

\section{The solution algorithm}
\label{sec:coupling_and_meshmovement}

In this section we describe how we iteratively solve the
discretization of the two-fluid model and the level-set equation, and
how we update the mesh in each space-time slab.

\subsection{Coupling discretization and mesh deformation}
\label{ss:coupling}

Given the level-set function $\phi_h$ from space-time slab
$\mathcal{E}^{n-1}$ we create an initial mesh for space-time slab
$\mathcal{E}^n$. Using Picard iterations we solve the space-time HDG
discretization \cref{eq:PicardNSdisc} for the momentum and mass
equations until some convergence criterium has been met. The velocity
solution to the space-time HDG discretization is then used in the
space-time EDG discretization \cref{eq:levelset_disc} to update the
level-set function which in turn is used to update the mesh. We
continue updating the mesh and solving the space-time HDG and EDG
discretizations in space-time slab $\mathcal{E}^n$ until the following
stopping criterium is met:
\begin{equation}
  \label{eq:stoppingcritphi}
  \frac{\norm[0]{\phi_h^{n,m} - \phi_h^{n,m-1}}_{\infty}}{\norm[0]{\phi_h^{n,m} - \phi_h^{n,0}}_{\infty}} < \varepsilon_{\phi},
\end{equation}
where $\varepsilon_{\phi}$ is a user given parameter and
$\phi_h^{n,m}$ is the approximation to $\phi_h$ after $m$ iterations
in the $n^{\text{th}}$ space-time slab. The algorithm is described in
Algorithm \ref{alg:coupling}.

\begin{algorithm}[H]
  \caption{Coupling the discretization and mesh deformation}
  \label{alg:coupling}
  \begin{algorithmic}[1]
    \State Initialize the flow properties and the level-set function
    \State Set $n = 0$, $t_n = 0$
    \While{$t_n < t_N$}
    \State Set $m = 0$
    \State Create an initial space-time mesh for space-time slab $\mathcal{E}^n$ given $\phi^{n,m}_h$
    \While{$\phi^{n,m}_h$ does not satisfy \cref{eq:stoppingcritphi}}
    \State Solve the Navier--Stokes \cref{eq:NSdisc} using Picard iterations \cref{eq:PicardNSdisc} to get $(u_h^{n,m+1}, p_h^{n,m+1})$
    \State Given $u_h^{n,m+1}$, solve the level-set equation \cref{eq:levelset_disc} to obtain $\phi^{n,m+1}$
    \State Modify the space-time mesh according to $\phi^{n,m+1}$
    \State Set $m=m+1$
    \EndWhile
    \State Set $u_h^n = u_h^{n,m+1}$, $p_h^n = p_h^{n,m+1}$, $\phi_h^n = \phi_h^{n,m+1}$, and $t_n = t_{n+1}$
    \EndWhile
  \end{algorithmic}
\end{algorithm}

\subsection{Mesh movement}
\label{ss:meshmovement}

Recall that the shape of the subdomains $\Omega_\ell(t)$ and
$\Omega_g(t)$ depends on the position of the free-surface
$\Gamma_s(t)$. Once the discrete level-set function $\phi_h$ is
obtained by solving \cref{eq:levelset_disc}, the wave height must be
obtained in order to update the mesh nodes. Traditionally, using
standard discontinuous Galerkin methods for free-surface problems, the
approximation to the wave height is discontinuous. This implies that
the free-surface of the domain is not well defined and a
post-processing of the free-surface is required to address this
\cite{Gagarina:2014, Vegt:2007}. This mesh smoothing, however, may
require extra stabilization terms (see \cite{Aizinger:2006}).

Using the space-time embedded discontinuous Galerkin method
\cref{eq:levelset_disc} for the level-set function, mesh smoothing is
not required. This is because the facet approximation to the level-set
function, $\bar{\phi}_h$, is continuous on the mesh skeleton, see
\cref{eq:barMhn}. We therefore avoid needing any mesh smoothing mesh
that may lead to instabilities while maintaining all the conservation
properties that discontinuous Galerkin methods provide.

We next describe how to obtain the wave height from the level-set
function and subsequently how to move the mesh nodes. We first note
that $\bar{\phi}_h$ is the trace of an $H^1$ function. Denoting by
$M_h^c$ the space of functions of $M_h$ which are continuous on
$\mathcal{E}^n$, we denote by $\phi_h^c$ the function in $M_h^c$ that
coincides with $\bar{\phi}_h$ on the element boundaries. We note that
it is computationally cheap to find $\phi_h^c$ because it can be found
element-wise. By definition $\phi((x_1,x_2),t) = \zeta(x_1,t) - x_2$ and
so an approximation to the wave height, $\zeta_h$, can be obtained by
evaluating $\phi_h^c$ at $x_2=0$.

Once we have obtained $\zeta_h$ we update mesh nodes as follows. Let
$(x_{1,i}^0, x_{2,i}^0)$ denote the coordinates of node $i$ of the
undisturbed mesh ($\zeta_h = 0$), and let $(x_{1,i}^k, x_{2,i}^k)$
denote the coordinates of the node $i$ at time $t_k$. Denote by
$T_b(x_1)$ ($B_b(x_1)$) the maximum (minimum) $x_2$ value in $\Omega$
on the vertical through $x_1$. Then:

\begin{itemize}
\item If $x_{2,i}^0 < 0$, 
  \begin{equation}
    x_{2,i}^{k+1} = x_{2,i}^0 + \gamma_i^k \zeta_h(x_{1,i}^k,t_{k+1})
    \quad \text{where }
    \gamma_i^k = \frac{B_b(x_{1,i}^k) + x_{2,i}^0}{B_b(x_{1,i}^k)}.
  \end{equation}
  
\item If $x_{2,i}^0 > 0$, 
  \begin{equation}
    x_{2,i}^{k+1} = x_{2,i}^0 + \gamma_i^k\zeta_h(x_{1,i}^k,t_{k+1})
    \quad \text{where }
    \gamma_i^k = \frac{T_b(x_{1,i}^k) - x_{2,i}^0}{T_b(x_{1,i}^k)}.
  \end{equation}	
\end{itemize}

\section{Numerical results}
\label{sec:numerical_results}

All the simulations in this section were implemented using the Modular
Finite Element Method (MFEM) library \cite{mfem-library}. Furthermore,
as is common with interior penalty type discretizations, we set the
penalty parameter to $\alpha = 10k^2$ \cite{Riviere:book}.

\subsection{Convergence rates}
\label{ss:con_rates}

We first consider a manufactured solutions test case to numerically
determine the convergence rates of our discretization. We consider a
domain $\Omega = [-1, 1] \times [-1,1]$. The exact pressure is given
by $p((x_1,x_2),t) = (2+\cos(t))\sin(\pi x_1) \cos(\pi x_2)$, whereas
the exact velocity on the liquid and gas regions, respectively, are
given by
\begin{equation}
	u_j((x_1,x_2),t) = \begin{bmatrix}
		\frac{0.1}{\nu_j}(e^{-t} - 1)\sin(\pi(x_2 - \zeta(x_1,t))) + 2
		\\
		\frac{0.1}{\nu_j}(e^{-t} - 1)\sin(\pi(x_2 - \zeta(x_1,t))) \partial_{x_1}\zeta(x_1,t) + 2
	\end{bmatrix}, \quad j = \ell, g.
\end{equation}
Note that $u_\ell$ and $u_g$ satisfy the interface conditions
\cref{eq:interfaceassump}. The source term in the momentum equation
\cref{eq:NS_mom}, and the level-set equation \cref{eq:NS_kin} are
computed according to the analytical solution. We impose Neumann
boundary conditions at $x_2 = 1$, while Dirichlet boundary conditions
are imposed everywhere else, and we take the polynomial degree
$k=2$. First, we test the convergence of the fluid solver with
discontinuous density and viscosity, and a fixed interface located at
$x_2 = 0$ (i.e., $\zeta(x_1,t) = 0$). We take $\mu_\ell = 1, \, \mu_g = 0.1$ and
$\rho_\ell = 1, \, \rho_g = 0.1$. \Cref{tab:ConvRates_fixedinterface}
shows the $L^2$ errors in the pressure and velocity at the final time
$t_N=1$. Since the mesh is fixed in this case, we do not show the
rates of convergence in the level-set function. We see that, as
expected, the pressure error is $\mathcal{O}(h^k)$, while the velocity
error is $\mathcal{O}(h^{k+1})$. These rates of convergence are
expected when compared to theoretical results of the single phase
Navier--Stokes problem \cite{Kirk:2019b}.

\begin{center}
  \begin{table}[tbp]
    \centering
    \caption{Rates of convergence for test case in
      \cref{ss:con_rates}, with $\mu_\ell = 1, \, \mu_g = 0.1$ and
      $\rho_\ell = 1, \, \rho_g = 0.1$ and $x_2 = 0$.}
    \begin{tabular}{cccccc}
      \toprule
      {Elements per slab} & $\Delta t$
          & {$\norm[0]{p-p_h}_{L^2(\Omega_N)}$} & {Rate}
                          & {$\norm[0]{u-u_h}_{L^2(\Omega_N)}$} & {Rate}\\
      \midrule
      {96} & {0.1} & {6.03e-1} & {-} & {1.01e-2} & {-}\\
      {384} & {0.05} & {1.34e-1} & {2.17} & {1.60e-3} & {2.66}\\
      {1,536} & {0.025}  & {3.46e-2} & {1.95} & {2.22e-4} & {2.85}\\
      {6,144} & {0.0125} & {8.92e-3} & {1.96} & {2.87e-5} & {2.95}\\
      {24,576} & {0.00625} & {2.38e-3} & {1.91} & {3.64e-6} & {2.98} \\
      \bottomrule
    \end{tabular}
    \label{tab:ConvRates_fixedinterface}
  \end{table}
\end{center}

We now consider the case where the interface between the gas and
liquid regions is time dependent. In particular, we consider the case
where the exact location of the interface satisfies
$\zeta(x_1,t) = 0.1 \sin(\pi(x_1-t))$ and choose the exact velocity
and pressure expressions as above. We take
$(\mu_\ell,\rho_{\ell}) = (1,1)$ in the liquid region. In
\cref{tab:ConvRates_movinginterface_0.1,tab:ConvRates_movinginterface_0.01}
we present the results for, respectively,
$(\mu_g, \rho_g) = (0.1, 0.1)$ and $(\mu_g, \rho_g) = (0.01, 0.01)$ in
the gas region. We see that for both cases, the pressure converges
linearly, while the velocity and the level-set converge
quadratically. We attribute this loss in accuracy with respect to
\cref{tab:ConvRates_fixedinterface} to the fact that the interface
between the gas and liquid regions, which is not a linear function, is
being represented with a mesh consisting of straight
elements. In~\cite{Huynh:2013, Wang:2013}, it is shown that there is a
loss in accuracy when super-parametric mesh elements are not used.

Finally, we remark that in all simulations, the divergence of the
approximate velocity is of the order of machine precision, even on
deforming meshes.

\begin{center}
  \begin{table}[tbp]
    \centering
    \small
    \caption{Rates of convergence for test case in
      \cref{ss:con_rates}, with $\mu_\ell = 1, \, \mu_g = 0.1$ and
      $\rho_\ell = 1, \, \rho_g = 0.1$ and
      $x_2 = 0.1 \sin(\pi(x_1-t))$.}
    \begin{tabular}{cccccccc}
      \toprule
      {Elements per slab} & $\Delta t$
      & {$\norm[0]{p-p_h}_{L^2(\Omega_N)}$} & {Rate}
      & {$\norm[0]{u-u_h}_{L^2(\Omega_N)}$} & {Rate} & {$\norm[0]{\phi-\phi_h}_{L^2(\Omega_N)}$} & {Rate}\\
      \midrule
      {96} & {0.1} & {7.01e-1} & {-} & {3.84e-2} & {-} & {4.71e-3} & {-} \\
      {384} & {0.05} & {1.64e-1} & {2.10} & {7.00e-3} & {2.46} & {9.01e-4} & {2.39} \\
      {1,536} & {0.025}  & {4.41e-2} & {1.89} & {1.16e-3} & {2.59} & {1.71e-4} & {2.40}\\
      {6,144} & {0.0125} & {1.24e-2} & {1.83} & {2.31e-4} & {2.33} & {3.65e-5} & {2.23}\\
      {24,576} & {0.00625} & {3.85e-3} & {1.69} & {5.26e-5} & {2.13} & {8.48e-6} & {2.11} \\
      \bottomrule
    \end{tabular}
    \label{tab:ConvRates_movinginterface_0.1}
  \end{table}
\end{center}

\begin{center}
  \begin{table}[tbp]
    \centering
    \small
    \caption{Rates of convergence for test case in
      \cref{ss:con_rates}, with $\mu_\ell = 1, \, \mu_g = 0.01$ and
      $\rho_\ell = 1, \, \rho_g = 0.01$ and
      $x_2 = 0.1 \sin(\pi(x_1-t))$.}
    \begin{tabular}{cccccccc}
      \toprule
      {Elements per slab} & $\Delta t$
      & {$\norm[0]{p-p_h}_{L^2(\Omega_N)}$} & {Rate}
      & {$\norm[0]{u-u_h}_{L^2(\Omega_N)}$} & {Rate} & {$\norm[0]{\phi-\phi_h}_{L^2(\Omega_N)}$} & {Rate}\\
      \midrule
      {96} & {0.1} & {7.44e-1} & {-} & {4.12e-1} & {-} & {3.90e-2} & {-} \\
      {384} & {0.05} & {1.77e-1} & {2.07} & {8.42e-2} & {2.29} & {8.23e-3} & {2.24} \\
      {1,536} & {0.025}  & {4.63e-2} & {1.93} & {1.53e-2} & {2.46} & {1.47e-3} & {2.49}\\
      {6,144} & {0.0125} & {1.33e-2} & {1.80} & {2.74e-3} & {2.48} & {2.48e-4} & {2.57}\\
      {24,576} & {0.00625} & {4.19e-3} & {1.67} & {6.12e-4} & {2.16} & {5.63e-5} & {2.14} \\
      \bottomrule
    \end{tabular}
    \label{tab:ConvRates_movinginterface_0.01}
  \end{table}
\end{center}

\subsection{Energy stability}
\label{ss:energy_stab}

We next test the energy stability property of our method. We consider
a domain $\Omega = [-1, 1] \times [-1,1]$ and a mesh that contains
6144 tetrahedra per slab. This corresponds to a total of 260352
degrees of freedom (after static condensation) for the flow problem
and 12675 degrees of freedom (after static condensation) for the
level-set equation. The final time of the simulation is $t_N = 10$, and we consider three different time steps
$\Delta t = 0.2$, $0.1$, and $0.05$. The set up of this test case is
similar to \cite[Section 5.2]{Horvath:2020}. In the first space-time
slab, we consider the source term for the two-fluid Navier--Stokes
equations to be an element-wise random number. For the following slabs
it is then set to be zero. Homogeneous Dirichlet and Neumann boundary
conditions are considered throughout the simulation. The level-set
equation is solved using the velocity obtained from the two-fluid
equations, with source term, initial and boundary conditions all set
to zero. Moreover, the free-surface moves according to the solution of
the level-set equation. We take $\mu_\ell = 1$, $\rho_\ell=1$,
$\mu_g = 0.01$ and $\rho_g = 0.01$. \Cref{fig:kin_energy} shows the
evolution of the kinetic energy
$\frac{1}{2}\norm[0]{\sqrt{\rho}u}_{L^2(\Omega_{n})}$, for
$n=1,\ldots,N-1$. We observe that the energy decays for all the time
steps considered, i.e., that
$\frac{1}{2} \partial_t(\norm[0]{\sqrt{\rho}u}_{L^2(\Omega_{n})}) \le
0$. This suggests that our method is energy-stable under mesh movement
for the two-fluid Navier--Stokes problem with discontinuous density
and viscosity. The formal proof of this property, however, is outside
the scope of this article.

\begin{figure}
  \centering
  \includegraphics[scale=0.8]{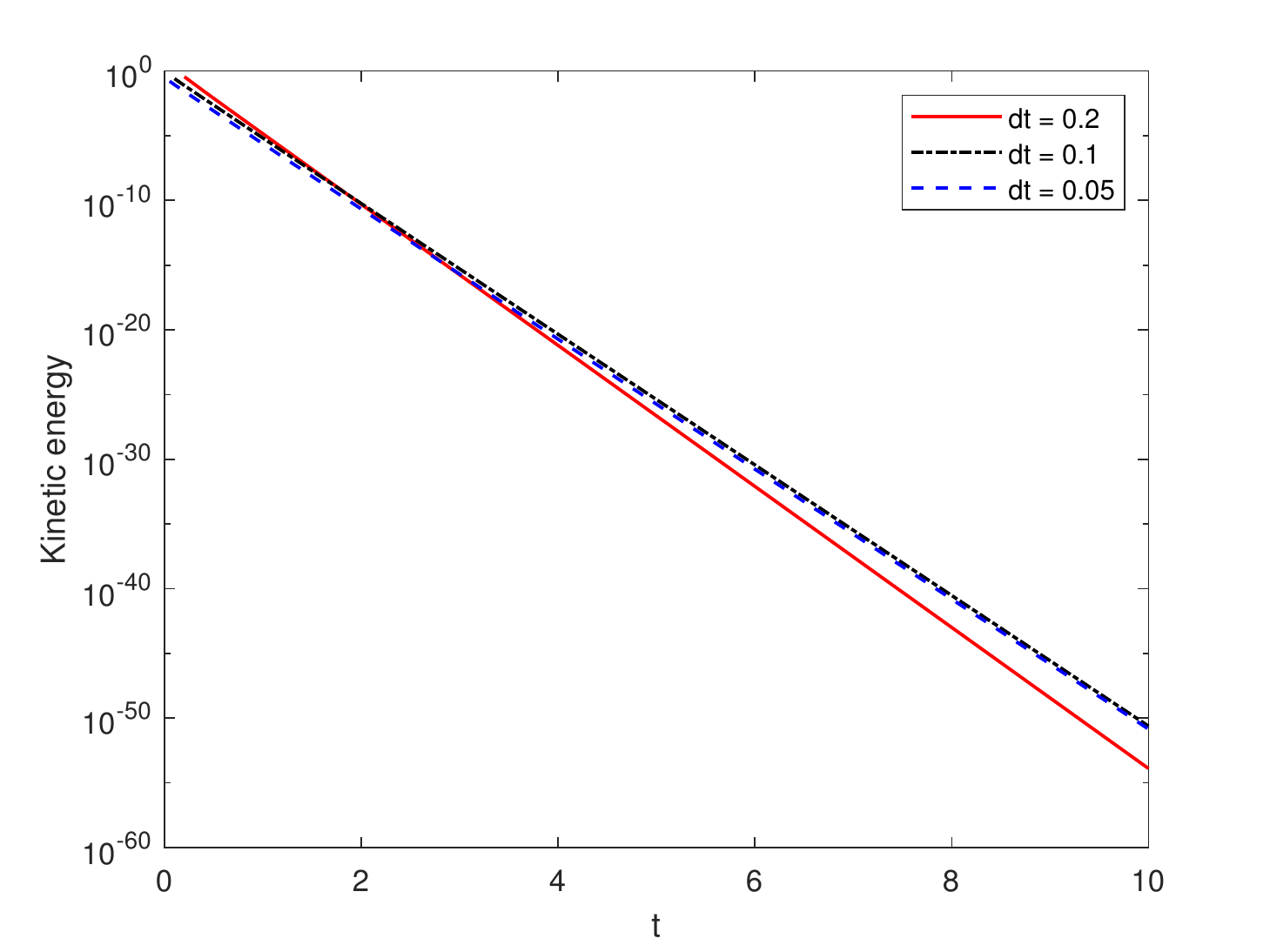}
  \caption{Evolution of the kinetic energy
    $\frac{1}{2}\norm[0]{\sqrt{\rho}u}_{L^2(\Omega_{n})}$ for
    different time steps, see \cref{ss:energy_stab}.}
  \label{fig:kin_energy}
\end{figure}

\subsection{Sloshing in a water tank}
\label{ss:sloshing}

We consider a small-amplitude periodic wave that is allowed to
oscillate freely in a rectangular tank with length that is twice the
depth of the still water level. The computational domain is
$\Omega = [-1,1]\times [-1,0.2]$. Initially the fluid is at rest and
the wave has a profile given by
\begin{equation}
\zeta_0(x_1) = 0.01\cos\del[0]{k(x_1 + 0.5)},
\end{equation}
where $k = 2\pi$ is the wave period. An analytical solution to the
linearized free-surface flow problem is given in \cite{Wu:2001}; given
a high enough Reynolds number and assuming a negligible influence of
the finite depth of the tank, the analytic wave height
$\zeta_{\text{ref}}$ is given by
\begin{equation}
  \frac{\zeta_{\text{ref}}(x_1,t)}{\zeta_0(x_1)}
  = 1 - \frac{1}{1+4\nu^2k^2/g}
  \sbr[2] {1 - e^{-2\nu k^2t}\del[2]{\cos(\sqrt{kg}t) +2\nu k^2\frac{\sin(\sqrt{kg}t)}{\sqrt{kg}}}},
  \label{eq:anal_sol_water_tank}
\end{equation}
where $\nu$ is the kinematic viscosity of the liquid which is set to
$\nu = 1/2000$. The density and viscosity ratios are both set to
$1:1000$. We consider two meshes, a structured mesh with 1152 spatial
triangles (3452 space-time tetrahedra), and a finer mesh with 2756
spatial triangles (8268 space-time tetrahedra) that is more refined
around the free-surface. At $x_2=-1$ we apply no-slip boundary
conditions, and at $x_2=0.2$ we apply a homogeneous Neumann boundary
conditions. The polynomial degree is $k = 2$. In \cref{fig:water_tank}
we compare the analytical solution of wave height elevation at the
middle of the tank ($x_1 = 0$) to the coarse and fine grid computed
approximations. On the coarse mesh the amplitude of the computed wave
height elevation dampens out as time progresses due to numerical
diffusion. On the fine mesh, however, the discrete wave height
elevation agrees well with the analytical solution.

\begin{figure}
  \centering
  \includegraphics[scale=0.3]{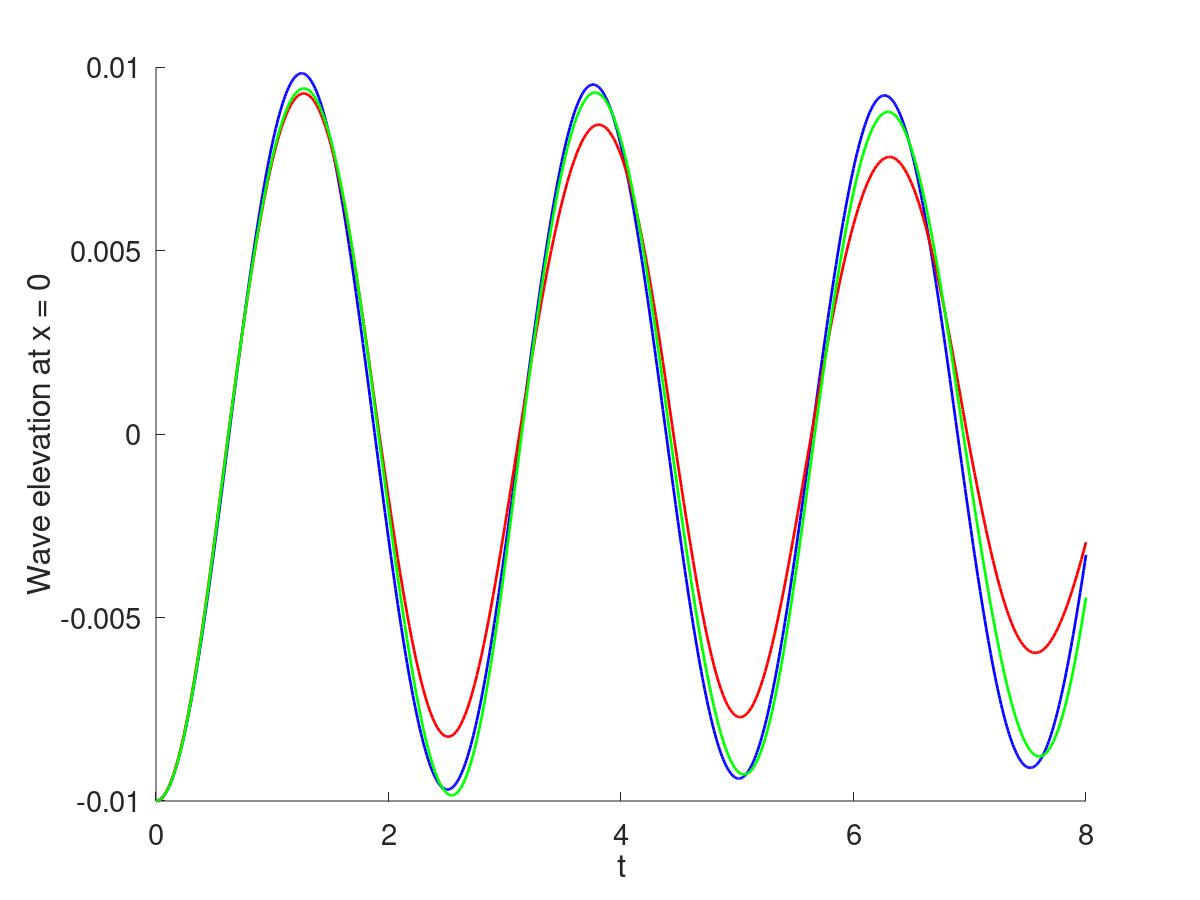}
  \caption{Wave elevation at $x_1 = 0$ with the anaytical solution
    \cref{eq:anal_sol_water_tank} (blue line), coarse mesh
    approximation (red line) and fine mesh approximation (green
    line).}
  \label{fig:water_tank}
\end{figure}

To show the mesh movement with the interface, we plot the mesh and
wave height at times $t = 0$ and $t = 3.86$ in
\cref{fig:mesh_waveheight}. A zoom of the mesh near the interface in
\cref{fig:mesh_waveheight}(c) and (d) show how the mesh conforms to
the interface.

\begin{figure}
	\centering
	\subfloat[Spatial mesh and wave height at $t = 0$.]
	{
		\includegraphics[width=\linewidth]{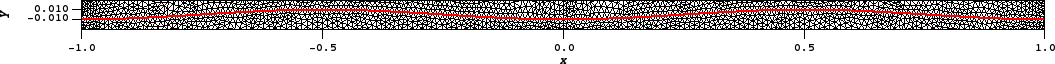}
	}
	\\
	\subfloat[Spatial mesh and wave height at $t = 3.86$]
	{
		\includegraphics[width=\linewidth]{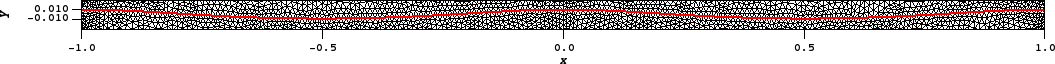}
	}
	\\
	\subfloat[Spatial mesh and wave height at $t = 0$.]
	{
		\includegraphics[width=0.5\linewidth]{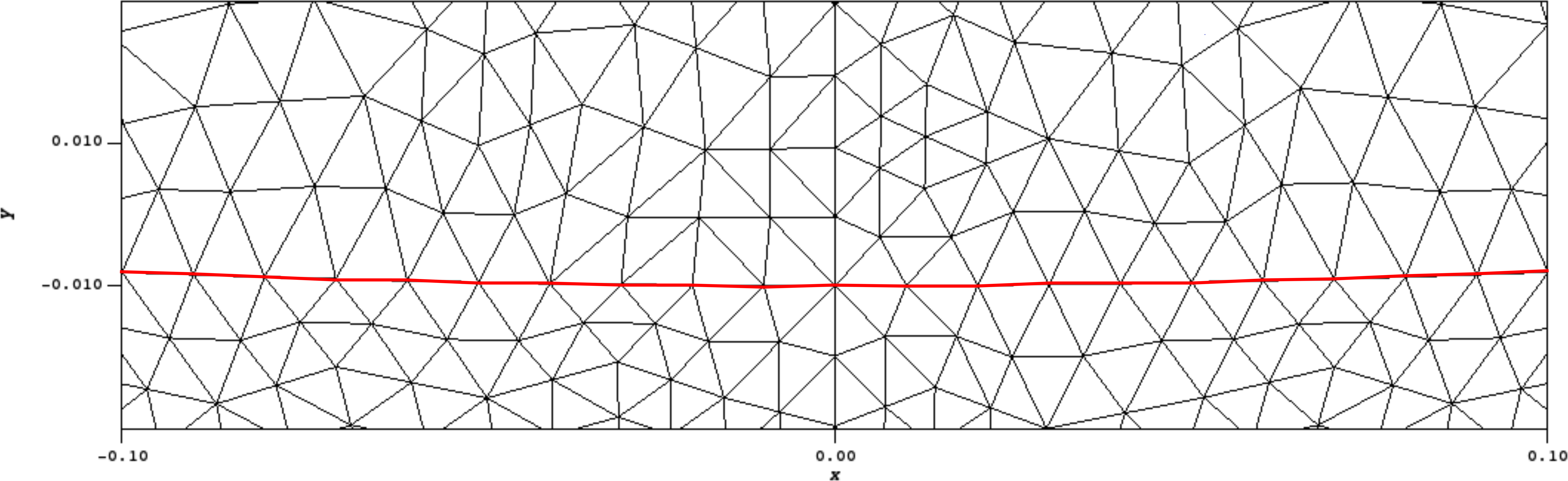}
	}
	\subfloat[Spatial mesh and wave height at $t = 3.86$]
	{
		\includegraphics[width=0.5\linewidth]{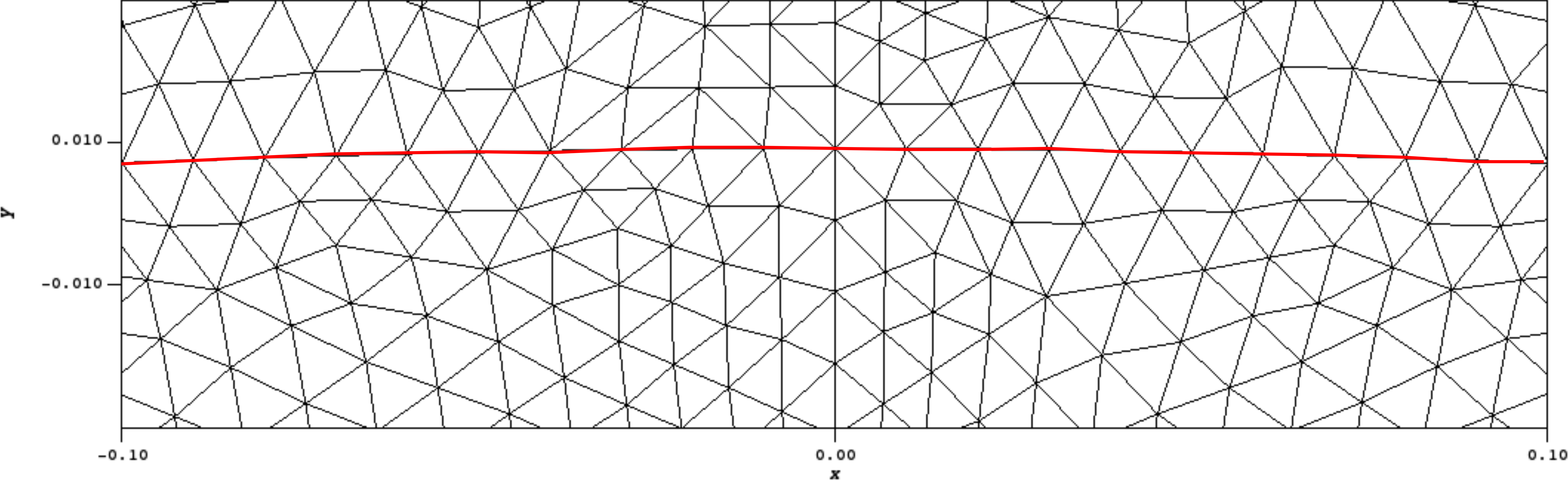}
	}  
	\caption{The spatial mesh at two instances in time for the test case
		described in \cref{ss:sloshing}. The top two figures are an
		extract of the mesh in $[-1,\,1]\times [-0.03,\,0.03]$. The bottom
		two figures zoom into the region
		$[-0.1,\,0.1]\times [-0.03,\,0.03]$. We indicate the wave height
		in all figures in red. Note that the mesh conforms to the
		interface.}
	\label{fig:mesh_waveheight}
\end{figure}
%

\subsection{Waves generated by a submerged obstacle}

In this final example, we consider waves in a channel generated by a
submerged cylinder. We consider a computational domain given by
$\Omega = [-8, 26]\times[-7,3]$ with a cylinder of radius $0.5$
located at $(0, -3)$ and the initial wave height set to $\zeta = 0$
(see \cref{fig:cylynder_dom}).
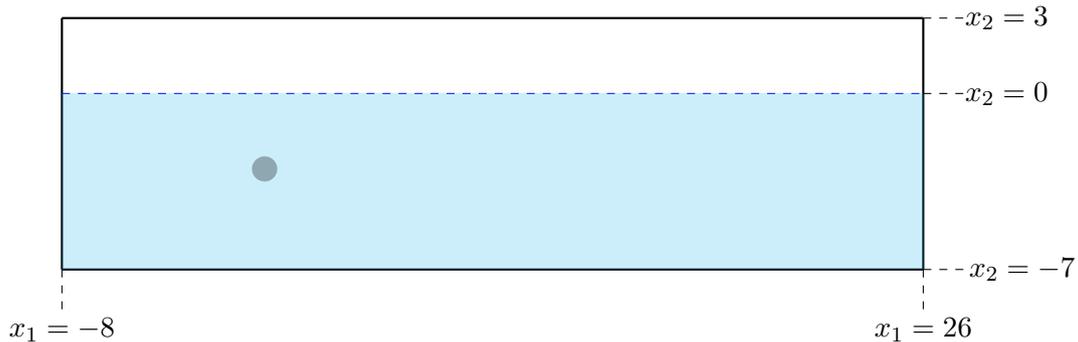
\begin{figure}[tbp]
  \begin{center}
    \begin{tikzpicture}[scale=1,important line/.style={thick}]      
      \draw[important line] (-8/3,-7/3) -- (26/3,-7/3) ;
      \draw[important line] (26/3,-7/3) -- (26/3,3/3) ;
      \draw[important line] (26/3,3/3) -- (-8/3,3/3) ;
      \draw[important line] (-8/3,3/3) -- (-8/3,-7/3) ;
      \draw[dashed, color=blue] (-8/3,0) -- (26/3,0);      
      \fill[color=cyan, fill opacity=0.2] (-8/3,-7/3) rectangle (26/3,0);      
      \fill[color = black, fill opacity = 0.3] (0, -3/3) circle (0.5/3);      
      \draw[dashed] (26/3,0) -- (26/3+0.6,0) ;
      \draw (26/3+1.1,0) node{$x_2 = 0$};      
      \draw[dashed] (26/3,3/3) -- (26/3+0.6,3/3) ;
      \draw (26/3+1.1,3/3) node{$x_2 = 3$};      
      \draw[dashed] (26/3,-7/3) -- (26/3+0.6,-7/3) ;
      \draw (26/3+1.3,-7/3) node{$x_2 = -7$};      
      \draw[dashed] (-8/3,-7/3) -- (-8/3,-7/3-0.6) ;
      \draw (-8/3,-7/3-0.8) node{$x_1 = -8$};      
      \draw[dashed] (26/3,-7/3) -- (26/3,-7/3-0.6) ;
      \draw (26/3,-7/3-0.8) node{$x_1 = 26$};
    \end{tikzpicture}    
  \end{center}
  \caption{Depiction of the flow domain
    $\Omega \subset \mathbb{R}^2$.}
  \label{fig:cylynder_dom}
\end{figure}
On the left, top and bottom boundaries of the domain, we impose
$\boldsymbol{u} = [0.54,0]^T$ whereas on the right boundary of the
domain we impose a homogeneous Neumann boundary condition. On the
boundary of the circle, homogeneous Dirichlet boundary conditions are
imposed. Moreover, $\boldsymbol{u}_0(x) = [0.54,0]^T$. For the
level-set function, we set $r = 0$ at the inflow part of the boundary
($x_1 = -8$). The density and viscosity ratios are
$\rho_g/\rho_\ell = 1/1000$ and $\mu_g/\mu_\ell = 0.01$. The time step
is taken as $\Delta t = 0.02$ and the space-time mesh contains 9312
tetrahedra.

In \cref{fig:cylinder} we show the velocity magnitude at time $t=6$ in
the liquid domain $\Omega_\ell$. We observe vortex shedding, as
observed also in single-phase flows (for example,
\cite{Schaefer:1996}), as well as waves being generated due to the
flow passing the obstacle.

\begin{figure}
  \centering
  \includegraphics[scale=0.8]{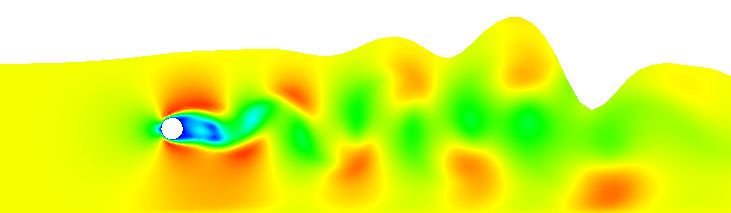}
  \caption{Velocity magnitude in $\Omega_\ell$ at time $t = 60$.}
  \label{fig:cylinder}
\end{figure}

\section{Conclusions}
\label{sec:conclusions}

We have presented a compatible, interface-tracking, exactly mass
conserving, space-time HDG/EDG discretization for the two-fluid
Navier--Stokes equations. The mesh moves with the zero-level set so
that the density and viscosity are always piece-wise constants with
discontinuity across the interface. Moreover, by using a space-time
EDG method for the level-set equation, we are able to obtain a
continuous approximation to the interface between the fluids; no
smoothing techniques are necessary to accommodate the mesh movement.

Our numerical simulations suggest that the method is energy-stable
under mesh movement, and that we can obtain optimal rates of
convergence considering that our mesh consists of straight
elements. Future work includes extending our discretization to using
curved elements near the interface to improve the accuracy of the
method and to extend our approach to more general two-fluid flow
problems such as rising bubbles in a column and fluid-structure
interaction problems.

\subsubsection*{Acknowledgements}

SR gratefully acknowledges support from the Natural Sciences and
Engineering Research Council of Canada through the Discovery Grant
program (RGPIN-05606-2015).

\bibliographystyle{elsarticle-num-names}
\bibliography{references}
\end{document}